\documentclass[12pt]{article}
\usepackage{amsmath,amssymb,latexsym}
\usepackage{amsfonts}
\parskip\bigskipamount
\newtheorem{theorem}{Theorem}[section]

\newtheorem{definition}{Definition}[section]

\newtheorem{lemma}{Lemma}[section]

\newenvironment{proof}[1][Proof]{\textbf{#1.} }{\ \rule{0.5em}{0.5em}}

\newcommand\E{{\mathbb E}}
\newcommand\N{{\mathbb N}}
\newcommand\R{{\mathbb R}}
\newcommand\Z{{\mathbb Z}}
\newcommand\ind[1]{\ensuremath{{\mathbf1}_{#1}}}
\newcommand\weaklyto{\stackrel{\mathcal D}\rightarrow}
\newcommand\abs[1]{\left\vert#1\right\vert}
\newcommand\given{\;\mid\;}
\newcommand\st{\;:\;}
\newcommand\midgiven{\;\vrule width 0.4pt\;} 

\newcommand\sA{{\mathcal A}}
\newcommand\sF{{\mathcal F}}

\newcommand\sL{{\mathcal L}}

\begin{document}
\title{\textbf{Information recovery from observations by a
random walk having jump distribution with exponential tails}}

\author{ANDREW HART%
\thanks{Departamento de Ingenier{\'\i}a Matem\'atica and Centro de Modelamiento 
Matem\'atico, UMI 2071 CNRS-UCHILE, Facultad de Ciencias F\'isicas y 
Matem\'aticas, Universidad de Chile, Casilla 170, Correo 3, Santiago, Chile}%
\thanks{E-mail:  \texttt{ahart@dim.uchile.cl}}
\and 
FABIO MACHADO
\thanks{Instituto de Matematica e Estatistica,
Universidade de São Paulo,
Rua do Matão 1010,
São Paulo, SP - Brazil} 
\and
HEINRICH MATZINGER%
\thanks{School of Mathematics,
Georgia Institute of Technology,
656 Cherrystreet,
Atlanta, 30363 GA,
USA}%
\thanks{E-mail: \texttt{matzi@math.gatech.edu}}%
\thanks{Corresponding author.}
}
\maketitle

\begin{abstract}
A  {\it scenery} is a coloring $\xi$ of the integers. Let
$\{S_t\}_{t\geq 0}$ be a recurrent random walk on the integers.
Observing the scenery  $\xi$ along the path of this random walk, one
sees the color $\chi_t:=\xi(S_t)$ at time $t$. The {\it scenery
reconstruction problem} is concerned with recovering the scenery
$\xi$, given only the sequence of observations $\chi:=(\chi_t)_{t\geq
0}$.  The scenery reconstruction methods presented to date
require the random walk to have bounded increments.  Here, we
present a new approach for random walks with unbounded increments
which works when the tail of the increment distribution decays
exponentially fast enough and the scenery has five colors.
\end{abstract}

{\footnotesize

}

\section{Introduction}

Consider a coloring of the integers $\xi:\Z\rightarrow\{1,2,3,4,5\}$ which we 
shall call a scenery. Let $S$ be a recurrent random walk starting at the origin. 
We assume that we observe the scenery along the path of the random walk $S$, 
that is, we observe the color $\chi_t:=\xi(S_t)$ at time $t$. The scenery 
reconstruction problem is concerned with determining the scenery $\xi$ using a 
single realization of  the color record
$$
\chi:=(\chi_0, \chi_1, \chi_2, \ldots).
$$
The scenery reconstruction problem can also be formulated as follows:  Does 
one path realization of the process $\{\chi_t\}_{t\geq 0}$ uniquely 
determine~$\xi$? The answer in such general terms is ``no''. However, under 
appropriate restrictions, the answer becomes ``yes''.   Firstly, we can at best 
hope to be able to reconstruct the scenery up to translation and reflection. 
Secondly, there are sceneries which can not be reconstructed. Lindenstrauss 
in~\cite{Lindenstrauss} exhibited sceneries which cannot be reconstructed.  Only 
`typical' sceneries can be reconstructed. Sceneries that can be obtained from 
one another by shift and reflection are called equivalent and we shall use the 
symbol $\approx$ to denote their equivalence.

The scenery reconstruction problem arose from questions posed by Kesten, Keane, 
Benjamini, Den Hollander and others.  It also falls into the research area 
concerned with the investigation of the ergodic properties of the color 
record~$\chi$. One of the motivations for studying scenery reconstruction comes 
from ergodic theory, for example via the $T,T^{-1}$problem; see 
Kalikow~\cite{Kalikow}. The ergodic properties of the observations $\chi $ were 
studied by Heicklen, Hoffman, Rudolph in \cite{Hoffman2000}, Kesten and Spitzer 
in \cite{KestenSpitzer}, Keane and den Hollander in \cite {KeaneDenHollander}, 
den Hollander in \cite{denHollander} and den Hollander and Steif in 
\cite{denHollanderSteif}.

A related important problem is the distinguishing of sceneries: Benjamini, den 
Hollander, and Keane independently asked whether all non-equivalent sceneries 
could be distinguished. We give a brief outline of this problem.  Let $\eta 
_{1}$ and $\eta _{2}$ be two given sceneries. Assume that either $\eta _{1}$ or 
$\eta _{2}$ is observed along a random walk path, but we do not know which one. 
Can we tell which of the two sceneries was observed? Kesten and Benjamini proved 
that one can distinguish almost every pair of sceneries, even in two dimensions 
and with only two colors. Before that, Howard had proved in \cite{Howard1}, 
\cite{Howard2}, and \cite{Howard3} that any two periodic one dimensional non-
equivalent sceneries are distinguishable, and that one can almost surely 
distinguish single defects in periodic sceneries. The problem of distinguishing 
two sceneries which differ only in one point is called ``detecting a single 
defect in a scenery''. Kesten in \cite{KestenSingDef} proved that one can a.s.\ 
recognize a single defect in a random scenery with at least five colors. He 
asked whether one can distinguish a single defect even if there are only two 
colors in the scenery.

Kesten's question was answered by Matzinger in his Ph.D.\ 
thesis~\cite{HeinisDiss} Given that the colors in the scenery are taken to be 
i.i.d.\ uniformly distributed, he showed that almost every 2-color scenery can 
be almost surely reconstructed up to equivalence. In \cite{Heini3color}, 
Matzinger proved that almost every 3-color scenery can be almost surely 
reconstructed. Kesten~\cite{KestenReview} noticed that the proofs employed 
in~\cite{HeinisDiss} and~\cite{Heini3color} rely heavily on the skip-free 
property of the random walk as well as the one-dimensionality of the scenery. He 
asked whether the result might still hold in more general situations.

In~\cite{Lowe-Matzinger-two-dim}, Matzinger and  L\"{o}we showed that one can 
still reconstruct sceneries in two dimensions, provided there are sufficiently 
many colors. In \cite{Matzinger-Rolles2001}, Rolles and Matzinger adapted the 
method proposed by L\"{o}we, Merkl and me to the case where random errors occur 
in the observed color record. They showed that the scenery can be reconstructed 
provided the probability of the errors is small enough. When the observations 
are seen with random errors, the reconstruction of sceneries is closely related 
to some coin tossing problems. These have been investigated by Harris and 
Keane~\cite{Harris-Keane97} and Levin, Pemantle and 
Peres~\cite {Levin-Pemantle-Peres2000}.

This paper deals with the problem of whether one can reconstruct a scenery seen 
along the path of a random walk having unbounded jumps, a question which was 
asked by den Hollander. Our main result shows that we can a.s. reconstruct a 
five-color (random) scenery seen along the path of a random walk $S$ with 
unbounded jumps, provided the probability of making a jump of non-unit size is not too 
high and the tail of the increment distribution of the random walk decays 
exponentially fast enough.  By unbounded jumps, we mean that the support of the 
random walk's increment distribution is not bounded. 

Methods  for carrying out Scenery reconstruction differ greatly depending on the distribution of 
the scenery and the nature  of the random walk. The methods for scenery reconstruction for  a 
simple random walk (\cite{HeinisDiss}, \cite{Heini3color},\cite{Heini2Color}, 
\cite{Lowe-Matzinger-two-dim}, \cite{Matzinger-Rolles-2color}, 
\cite{Matzinger-Rolles-polynomial}, \cite{Matzinger-Rolles2001}), together with the methods 
appropriate when the support of the increment distribution is bounded 
(\cite{Lowe-Matzinger-Merkl2001} and \cite{messagetext}), fail in the setting of 
``unbounded jumps''. It was not possible to adapt existing methods to the 
present situation. The approach utilized here is fundamentally different from 
those which have been fruitfully applied in the setting of random walks having ``bounded 
jumps''.

We begin by explaining the setting considered in  this article.
Let $c>0$ be an exponential decay rate and $\epsilon>0$ the
probability of jumping a distance that is \textit{not} of unit length. We 
shall impose the following conditions on the random walk $S$:
\begin{equation}
\label{conditionS1}
P(\abs{S_{t+1}-S_t}\neq 1)= \epsilon
\end{equation}
and
\begin{equation}
\label{conditionS2}
P(\abs{S_{t+1}-S_t}=i\given \abs{S_{t+1}-S_t}\neq 1) \leq e^{-ci}.
\end{equation}
We shall also require that the distribution of the random walk be symmetric:
\begin{equation}
\label{conditionS3}
P(S_{t+1}-S_t=i)=P(S_{t+1}-S_t=-i).
\end{equation}
This condition could be replace by the weaker condition $\E[S_{t+1}-
S_t]=0$, but symmetry simplifies notation. Finally, to
ensure that the random walk is non-periodic, we shall stipulate that
\begin{equation}
\label{conditionS4}
P(S_{t+1}-S_t=0)>0.
\end{equation}

We can now formulate the main theorem of the paper.

\begin{theorem}
\label{maintheorem}
Assume the scenery $\xi:\Z\rightarrow \{1,2,3,4,5\}$ is i.i.d. with
the five colors uniformly distributed.  If Conditions
(\ref{conditionS1})--(\ref{conditionS4}) are all satisfied, then for
every $\epsilon>0$ small enough and $c>0$ large enough, we can a.s.
reconstruct the scenery~$\xi$. In other words, there exists a measurable map 
$$
\mathcal{A}:\{1,2,3,4,5\}^\N \rightarrow \{1,2,3,4,5\}^\Z
$$
such that
$$
P(\mathcal{A}\circ \chi \approx \xi)=1.
$$
\end{theorem}

The principal difficulty with proving the theorem is in
showing that we can reconstruct the scenery at one point
given that we have already managed to reconstruct the scenery
on an interval. The next three sections of the paper are dedicated
to obtaining an estimate~$\hat\xi_{n+1}$ of~$\xi(n+1)$ given the observations~$\chi$ and the restriction
of the scenery $\xi\vert_{[-n,n]}:=\xi_{-n}\xi_{-n+1}\cdots \xi_{n-1}\xi_n$.
Section~\ref{sec:simpleprobs} uses three simplified problems to introduce key concepts that shall be required in the sequel.
In Section~\ref{sec:singlepointreconstruction}, we describe the algorithm for retrieving $\xi_{n+1}$ from~$\chi$ given that we know a portion $\xi\vert_{[-n,n]}$ of the scenery.  The following section then proves that the single-point reconstruction algorithm succeeds, that is, $\hat\xi_{n+1}=\xi_{n+1}$, with high probability.  This is done by showing that the probability of failure is finitely summable, that is,
\begin{equation}
\label{finitesummability}
\sum_n P\bigl(\xi(n+1)\neq \hat\xi(n+1)\bigr)<\infty.
\end{equation}
Finally, we conclude by showing in Section~\ref{wholereconstruction} that
finite summability implies that the whole scenery~$\xi$ can be reconstructed almost surely up to equivalence.

\section{Three simplified problems}
\label{sec:simpleprobs}

To illustrate the key concepts needed for  performing scenery
reconstruction on 5-color sceneries observed by random walks with
unbounded jumps, we begin by presenting three simplified problems
and their solutions. Each solution highlights one key idea
which we shall use later.

\subsection{Reconstructing one point when the scenery is observed at i.i.d. locations}

Take any non random scenery $\xi:\N\rightarrow \{1,2,3,4,5\}$. Let
$a$ and $b$ be two integers with $a<b$ and define $I:=[a,b]$. Let
$Y_1,Y_2,\ldots$ be i.i.d. random variables such that
\begin{equation} 
\label{conditionforreconstrutpoint}
P(Y_i=b+1)>P(Y_i\notin[a,b+1]).
\end{equation}
The first simple problem we consider, is to reconstruct the scenery $\xi$
at a single point, namely $b+1$. 
So we need to determine the value of $\xi(b+1)$. For
this, we suppose we are given two things: The restriction of $\xi$ to
the interval $I$, that is, $\xi\vert_I:=(\xi_i\, i\in I)$, and an infinite
sequence of observations of the scenery~$\xi$ at the random
locations $Y_i$,
\begin{equation}
\label{observedsequence}
\xi(Y_1),\xi(Y_2),\ldots.
\end{equation}
Now, for any $e\in\{1,2,3,4,5\}$, we have
$$
P(\xi(Y_i)=e)=P(\xi(Y_i)=e,Y_i\in[a,b]) 
+ \delta_{\xi(b+1)=e}P(Y_i=b+1) + P(Y_i\notin[a,b+1],\xi(Y_i)=e).
$$
Let $q_e$ be the quantity 
\begin{equation}
\label{qe}
q_e:=P(\xi(Y_i)=e)-P(\xi(Y_i)=e,Y_i\in[a,b]).
\end{equation}
When inequality \ref{conditionforreconstrutpoint} holds, two
possibilities manifest. If $\xi(b+1)=e$, then $q_e\geq P(Y_i=b+1)$.
On the other hand, if $\xi(b+1)\neq e$, then $q_e< P(Y_i=b+1)$. We
can use this dichotomy to reconstruct $\xi(b+1)$ as follows: \\
Estimate $q_e$ and take the element $e\in\{1,2,3,4,5\}$ which maximizes it for the color  at point
$b+1$.

   To determine $q_e$, we first estimate the probability
$P(\xi(Y_i)=e)$ from the sequence \ref{observedsequence}.
The probability $P(\xi(Y_i)=e,Y_i\in[a,b])$ can be calculated,
since we  are given the restriction of $\xi$
to $I$. For this, we also assume that the distribution
of the $Y_i$'s is known to us.

Note that Condition~\ref{conditionforreconstrutpoint}
 holds automatically when the tail of $Y_i$ decays exponentially
with decay rate $r$ strictly less than $1/2$ and
\begin{equation}
\label{conditiononab}
a<0<b \mbox{ and } |a|\gg|b|.
\end{equation}
To see this, let the $Y_i$'s have a symmetric distribution with
\begin{equation}
\label{exponentialdecaycondition}
P(Y_i=y)\cdot r\geq P(Y_i=y+1)\;\;,\;\;\forall y\in\N,
\end{equation}
where $r<0.5$. Assuming that $|a|\gg |b|$, $P(Y_i<a)$ is
negligible in comparison to $P(Y_i=b+1)$. Furthermore, thanks to the
exponential decay, we obtain
$$
\frac{P(Y_i>b+1)}{P(Y_i=b+1)}\leq \frac{r}{1-r}<1.
$$
Finally, if $P(Y_i\leq a)$ is small enough then this last inequality
implies Condition~\ref{conditionforreconstrutpoint}.

\subsection{Reconstructing a point when the scenery is seen 
along an infinite number of random walks}
\label{sec:ex2}

Once again assume that $\xi$ is a non-random scenery. As in the
previous setting, the restriction of $\xi$ to $[a,b]$ is known and
we try to reconstruct $\xi(b+1)$. We also assume that condition
\ref{conditiononab} holds. Let
$$
\{S_t^1\}_{t\in\N},\{S_t^2\}_{t\in\N},\ldots
$$
be a sequence  of random walks independent of each other and
all of which start at the origin.  Each of the $S^i_t$'s is identical to the random walk $S_t$, which we define to have the following increment distribution:
\begin{equation}
\label{simplestincrementdistribution}
P(S_t-S_{t-1} = x)
= \left\{\begin{array}{ll}
\epsilon, \mbox{if } x=0, \\
\gamma, & \mbox{if } \abs x=1, \\
0, & \mbox{otherwise,}
\end{array}\right.
\end{equation}
where $\gamma=\frac{1-\epsilon}2$ and $t\in\N$.  Thus, each $S_t^i$ is a simple 
symmetric random walk modified to allow sojourns of duration greater than unity 
in any state.  This is the simplest random walk satisfying Conditions~\ref{conditionS1}--\ref{conditionS4}. Let $\chi^i$ denote the 
observations of the scenery~$\xi$ made by random walk number $i$:
$$
\chi^i:= (\chi_i^t,\ t=0,1,2,\ldots),
$$
where $\chi^i_t:=\xi(S_t^i)$. In addition to knowing the restriction 
$\xi\vert_{[a,b]}$, we assume that all the observations 
$\chi^1,\chi^2,\chi^3,\ldots$ made by the different random walks are given to 
us.

Take $q$ such that $1<q<3$. In order to apply the reconstruction technique from 
the previous section, we set $Y_i$ to be equal to $Y_i:=S^i_r$, where $r=qb$ is 
an integer. The state distribution of the symmetric random walk $S_t$ is given 
by
$$
P(S_t=x) = \left\{\begin{array}{ll}
\sum_{s\in[0,t-x] \st s+t \mbox{ is even}} \binom ts\binom{t-s}{(t-s+x)/2} \epsilon^s\gamma^{t-s}, & \mbox{if } \abs x\leq t, \\
0, & \mbox{otherwise.}
\end{array}\right.
$$
Provided $\abs x\leq t$, the decay rate of $S_t$ at the point~$x$ is given by
\begin{eqnarray*}
\rho(t,x) &:=& \frac{P(S_t=x+1)}{P(S_t=x)}
= \frac{\sum_{s\in[0,t-x-1] \st s+t \mbox{ is even}} \binom ts\binom{t-s}{(t-s+x+1)/2} \epsilon^s\gamma^{t-s}}{\sum_{s\in[0,t-x] \st s+t \mbox{ is even}} \binom ts\binom{t-s}{(t-s+x)/2} \epsilon^s\gamma^{t-s}} \\
&\leq& \frac{\sum_{s\in[0,t-x] \st s+t \mbox{ is even}} \binom ts\binom{t-s}{(t-s+x+1)/2} \frac{(t-s-x)/2}{(t-s+x+1)/2} \epsilon^s\gamma^{t-s}}{\sum_{s\in[0,t-x] \st s+t \mbox{ is even}} \binom ts\binom{t-s}{(t-s+x)/2} \epsilon^s\gamma^{t-s}} \\
&\leq& \max_{s\in[0,t-x]} \frac{t-s-x}{t-s+x+1}
= \frac{t-x}{t+x+1}.
\end{eqnarray*}
Then, 
\begin{equation}
\label{decayrate}
\rho(r,b) = \rho(qb,b) = \frac{q-1}{q+1+1/b} < \frac{q-1}{q+1}.
\end{equation}
For $q\in(1,3)$, The final expression is always less than
$1/2$ and, since~$q$ does not depend on~$b$, the tail of $S_r$ decays
at a rate less than $1/2$ beyond the point $b$. Hence, condition
\ref{exponentialdecaycondition} is satisfied and we can apply the
reconstruction technique described in the previous subsection.

	So far, we have described how to perform a one-point reconstruction
when presented with an infinite set of realizations of the random
walk~$S$.  However, in the problem under consideration in this
paper, we only have access to the observations made by a single
random walk $S$. We get around this limitation by using stopping
times to restart the random walk in a predetermined way.  This
technique yields the infinite set of realizations we require.  Let
$\tau_i$ denote the time of the $i$-th visit by~$S$ to
the origin. Then due to the strong Markov property, the random walk
after time $\tau_i$ behaves like a random walk starting at the
origin. Hence, if we are given the piece of scenery $\xi\vert_{[a,b]}$
together with  observations $\chi$ and the sequence of stopping
times $\tau_1,\tau_2,\ldots$, then we can use the reconstruction
technique for an infinite number of random walks. For this we simply
take the sequence of observations $\chi^i$ made by the $i-th$ random
walk to be
$$
\chi_{\tau_i}, \chi_{\tau_i+1}, \chi_{\tau_i+2}, \ldots, \chi_{\tau_{i+1}-1}.
$$

\subsection{Finding the way back to the origin}

Because we are only given the sequence of observations made by the
random walk, we are not able to determine when it is at the origin and hence the stopping times
$\tau_1,\tau_2,\ldots$ are not observable.  Instead, we must
construct stopping times based on the observations which are able to
stop the random walk close to some point of reference.

Let $K=[k_1,k_2]\subset[-n,n]$ be an integer interval of length $n$.
The precise position of $K$ within $[-n,n]$ is not 
important here and will be specified later.  Our point of reference will be
the location of the finite string
$w:=\xi_{k_1}\xi_{k_1+1}\ldots\xi_{k_2}$. Let $\tau_i$ be the $i$-th
time that we observe the pattern $w$ in the observations $\chi$.
Hence 
\begin{equation}
\label{tau1}
\tau_1:=\min\{ t\geq n \st \chi_{(t-n)}\chi_{(t-n+1)}\ldots\chi_{(t)}=w\}
\end{equation}
and
\begin{equation}
\label{taui}
\tau_{i+1} := \min\{t>\tau_i \st \chi_{(t-n)}\chi_{(t-n+1)}\ldots\chi_{(t})=w\}.
\end{equation}
Let $R$ be a map $R:[0,n]\rightarrow\Z$. We call $R$ a {\it simple
random walk path}, if
$$
\forall t\in[0,n-1] \;,\; \abs{R(t+1)-R(t)}=1.
$$
If $R(0)=x$, we say the path $R$ starts at $x$. We assume that the
scenery $\xi$ is random, i.i.d. and all five colors appear with
equal probability. Let $x\notin[-2n,2n]$. Let $R$ be a non-random
simple random walk path starting at $x$. What is the probability
that $\xi$ seen along the path $R$ gives the string $w$?  In other
words, what is the probability that 
$$
\xi\circ R := \xi(R_0)\xi(R_1)\ldots\xi(R_n)=w?
$$
Note that since $x\notin[-2n,2n]$, the path of $R$ can not reach $[-
n,n]$. Since the scenery is i.i.d. and since $w$ only depends on
$\xi\vert_{[-n,n]}$, the pattern $w$ is independent of $\xi\circ R$. The
five colors in the scenery having the same probability, we find
$$
P(\xi\circ R=w)=\left(\frac15\right)^{n+1}.
$$
There are $2^n$ simple random walk paths starting at $x$ of length
$n$. Hence the probability that there exists such a path generating 
the color record $w$ is bounded:
$$
P(\exists R \mbox{ a simple r.w. path starting at $x$ such that }
 \xi\circ R=w) \leq \left(\frac15\right)^{n+1}\cdot 2^n<
\left(\frac25\right)^n.
$$

It follows that with probability close to one, every time we observe
the pattern~$w$ before the time $T=\rho^n$, the random walk must be in
the interval $[-2n,2n]$:
$$
S_{\tau_i}\in[-2n,2n]\mbox{ for every }\tau_i\leq T,
$$
where $\rho>0$ is a constant not depending on $n$ such that
$\rho<(5/2)^2$.

We shall see that this argument can be refined so that the interval
in which $S_{\tau_i}$ lies is much narrower. Also, we shall be
dealing with non-simple random walks. Hence it will be necessary to
adapt the present argument to that situation.

\section{Single Point Reconstruction} 
\label{sec:singlepointreconstruction}

In this section, we describe the algorithm for reconstructing $\xi(n+1)$
given the observations~$\chi$ and the restriction
$\xi\vert_{[-n,n]}$.
Before beginning, it is useful to consider a small numerical example to illustrate.

\subsection{ A numerical example}

Assume that the random walk is simple and suppose we are given the
finite restriction of $\xi$ to the interval $I:=[-4,4]$:
$$
\begin{array}{c|cr|r|r|r|c|c|c|c|c|c}
\xi(z)& & 2&4 &3 &2 &4&5&1&5&3&? \\
\hline
z     & &-4&-3&-2&-1&0&1&2&3&4&5
\end{array}
$$
We wish to reconstruct the value $\xi(5)$. Let $K:=[0,3]$ and let
$w$ be the restriction of $\xi$ to $K$:
$$
w:=4515.
$$
Occurrences of the pattern~$w$ in the observations~$\chi$ define a
sequence of stopping times $(\tau_i)$ like that defined by
(\ref{tau1}) and (\ref{taui}). Note that if we observe the pattern
$w$ and this was generated while the random walk was in the interval
$I$, then we must be located at either $z=1$ or $z=3$. More
precisely, if $S_s\in I$ for all $s\in[\tau_i-3,\tau_i]$, then
$S_{\tau_i}\in\{1,3\}$ and the random walk will be at one of the
points $1$ or $3$ with equal probability. Let $\mu$ denote the two-
atom measure which accords probability $1/2$ to $\{1\}$ as well as
to $\{3\}$. Let $P_{\mu}(.)$ denote the probability distribution for
the random walk $S$ starting with initial distribution $\mu$ instead
of starting at the origin.

Take $r=4$. As was done in (\ref{qe}), we define
\begin{equation}
q_e:=P_{\mu}\left(\xi(S_r)=e\right)-
P_\mu\left(\xi(S_r)=e,\xi(S_r)\in [-4,4]\right).
\end{equation} 
Then, we look at the empirical frequency of the colors at a fixed
offset~$r$ following each stopping time	 $\tau_i$. More
precisely, we estimate $P_\mu(\xi(S_r)=e)=P_\mu(\chi_r=e)$ by the
empirical distribution of
$$
\chi_{\tau_1+r},\chi_{\tau_2+r},\ldots,\chi_{\tau_j+r},
$$
 where
$j$ is the largest integer $i$ such that $\tau_i\leq T$.
Let us denote this empirical distribution by
\begin{equation}
\label{hatP}
\hat{P}_\mu(\chi_r=\cdot).
\end{equation}
We can now describe the reconstruction algorithm.

\smallskip\noindent
\textbf{Algorithm for single point reconstruction.}
Given $\xi\vert_{[-4,4]}$ and $\chi$, our estimate of $\xi(5)$ is the element $e$ of the set
$\{1,2,3,4,5\}$ which maximizes the quantity
$$
\hat{q}_e:=\hat{P}_\mu(\chi_r=e)-P_\mu\left(\xi(S_r)=e,S_r\in [-4,4] \right).
$$
\smallskip

Note that the second term in the definition of $\hat q_e$ can be 
calculated explicitly because $\xi\vert_{[-4,4]}$ is given.

To verify that the above algorithm has a high probability of
estimating $\xi(5)$ correctly, we first need to check that
\begin{equation}
\label{conditionexample}
P_\mu(S_r=5)>P_\mu(S_t\notin[-4,5]).
\end{equation}
This inequality corresponds to Condition 
\ref{conditionforreconstrutpoint}.
Taking $r=4$ we find
$$
P_\mu(S_r=5)=\frac5{32}, 
\qquad P_\mu(S_r\in (-\infty,-5]\cup[6,\infty))
=\frac1{32}
$$
and so inequality \ref{conditionexample} is satisfied.

The second problem we need to take care of is verifying that the
estimate \ref{hatP} is precise enough with high probability. For the
algorithm to work, the estimation error needs to be strictly less
than
\begin{equation}
\label{boundforerror}
\frac{P_\mu(S_r=5)-P_\mu(S_r\notin[-4,5])}2.
\end{equation}
For this we need to show that there exists~$T$ such that the
following two conditions are satisfied with high probability:
	
\begin{itemize}
\item
The time $T$ is large enough so that
 there are enough visits by the random walk to the origin
generating the pattern $w$ up to time $T$.
\item
The time $T$ is not too large, since otherwise during 
the time interval $[0,T]$ the random walk will wander too far away
from the origin. Far from the origin there are other places
where the pattern $w$ can be generated. Thus, too large a $T$
results in $S_{\tau_i}$ being distant from the origin at some of the times $\tau_i\in[0,T]$.
\end{itemize}

In the numerical example used to illustrate above, $n$ has been taken
too small for the required estimate precision to be obtained. One of the main issues
in the general context is to show that for $n$ large enough,
we can find a~$T$ such that with high probability the two 
 conditions above can be satisfied simultaneously. 

\subsection{The one-point reconstruction algorithm}
\label{thealgorithm}

Now we present an algorithm
for estimating $\xi(b+1)$ given the restriction $\xi\vert_{[-n,n]}$ and
the observations $\chi$.

Unlike the numerical example above, our problem is
to reconstruct the scenery when the random walk is not simple.
for this we shall need the concept of a $\delta$-path.
The key feature of a $\delta$-path is that it behaves like a simple
random walk most of the time and the proportion~$\delta$ of time
that it doesn't do so, its freedom of movement is linearly
restricted pro rata.

\begin{definition}
Let $\delta>0$. We call $R:[0,n]\rightarrow \Z$ a $\delta$-path
(of length $n$) if the proportion of steps of non-unit length as well as
 their total variation is less than $\delta$, that is,
$$
\abs M \leq \delta n
$$
and
$$
\sum_{t\in M} \abs{R(t+1)-R(t)} \leq \delta n,
$$
where the set $M$ is defined by
$$
M:=\{t\in[0,n-1] \st \abs{R(t+1)-R(t)} \neq 1\}.
$$
\end{definition}

For the Reconstruction, we shall use the stopping times $\tau_i$ that occur prior to time $T=2.4^{2n}$.
We will subsequently show that by taking the parameter
$\epsilon>0$ small enough, the random walk can be made to follow $\delta $-paths for any time interval
of length $n$ before time~$T$.

Our algorithm for reconstructing $\xi(n+1)$ is defined using three
intervals. First
$$
I:=[-n,n]
$$
is the interval on which we already know the scenery $\xi$.  Let 
$$
n^*:=n-61\delta n
$$
where $\delta>0$ is a small constant not depending on $n$.  The second interval
$$
K=[k_1,k_2]:=[n^*-n,n^*]
$$
is the interval from which we take the pattern
$$
w:=\xi(k_1)\xi(k_1+1)\xi(k_1+2)\ldots\xi(k_2).
$$
Finally we will show that after reading the pattern $w$ close to the
origin we are typically located in the interval 
$$
J:=[n^*-21\delta n,n^*+\delta n].
$$

Next, let $\nu_i$ be the $i$-th time the random walk generates
the pattern $w$ on the piece of scenery $\xi\vert_{[-n,n]}$:
\begin{equation}
\label{nu1}
\nu_1:=\min\{ t\geq n \st \chi_{t-n}\chi_{t-n+1}
\ldots\chi_t=w\;;\;
\forall s\in[t-n,t],\; S_s\in[-n,n]
\}
\end{equation}
and 
\begin{equation}
\label{nu_i}
\nu_{i+1}:=\min\{t>\nu_i \st \chi_{t-n}\chi_{t-n+1}\ldots
\chi_t=w\;
,\;
S_s\in[-n,n],\forall s\in[t-n,t]
\},
\qquad i\geq1.
\end{equation}
The difference between the $\nu_i$'s and the $\tau_i$'s is the
additional constraint on the $\nu_i$'s that the pattern should be
generated while the random walk is in $[-n,n]$. This is apriori not
observable. But we will show that the $\tau_i$'s and the $\nu_i$'s
coincide up to time~$T$ with high probability.

Note that given~$\xi$, the sequence
$$
S_{\nu_1},S_{\nu_2},S_{\nu_3},\ldots 
$$
is a Markov chain whose state space is $[-n,n]$. The transition
probability from $x$ to $y$ is given by 
$$
P_{xy}:=P(S_{\nu_{i+1}}=y \given S_{\nu_i}=x,\xi).
$$

This chain is aperiodic and irreducible. Denote its stationary
distribution by~$\mu$. Hence,~$\mu$ depends on~$\xi$ and is thus a
random measure. As before, $P_\mu(\cdot)$ denotes the measure for
which the random walk $S$ has starting distribution $\mu$ and for
which the distribution of the scenery remains unchanged and is
independent of the random walk. We take $r:=90\delta n$ and look at
the frequency of the observed color~$r$ time steps after a stopping
time. Our estimate for the distribution $\sL_\mu\left(\chi_r \right)$
is the empirical distribution of
$$
\xi(S_{\tau_i+r}),\xi(S_{\tau_2+r}), \ldots ,\xi(S_{\tau_j+r}),
$$
where $j$ denotes the largest integer for which $\tau_i\leq T$.
We shall denote this empirical distribution by 
$$
\hat{P}_\mu(\chi_r=\cdot).
$$

We are now ready to describe the reconstruction algorithm for estimating the value $\xi(n+1)$.

\smallskip\noindent
{\bf Algorithm:}
Given $\xi\vert_{[-n,n]}$ and $\chi$, the estimate $\hat{\xi}(n+1)$ of
$\xi(n+1)$ is the element $e \in\{1,2,3,4,5\}$ which maximizes the
quantity
$$
\hat{q}_e:=\hat{P}_\mu(\chi_r=e)-P_{\mu}\left(\xi(S_r)=e\;,\;S_r\in I\right).
$$
Note that the second term on the right-hand side of the definition of $\hat q_e$ can be
calculated explicitly since we know $\xi\vert_I$.

\subsection{Combinatorial aspects of the one point Reconstruction algorithm}

Let $A^n$ denote the event that the above reconstruction algorithm
works correctly, that is the event that $\hat{\xi}(n+1)=\xi(n+1)$.
We list a few events which are important in making the reconstruction
of $\xi(n+1)$ work.

\begin{description}
\item[$B^n$:]
Performing reconstruction will require the random walk to
remain near the origin, in some sense, for a sufficiently long time.
To capture this requirement, let $B^n$ be the event that the random
walk stays in the interval  $[-2.45^n,2.45^n]$ for all times prior to
and including $T=2.4^{2n}$.

\item[$C^n$:]
In order to obtain sufficiently precise estimates, we will need
enough observations close to the origin. This is taken care of by
$C^n$ which is the event that at least $(1.1)^n$ stopping times
$\nu_i$ occur before time~$T$:
$$
C^n:=\{\nu_i\leq T	\st i\leq 1.1^n\}.
$$

\item[$D^n$:]
We take $D^n$ to be The event that, up to time $T$, all pieces of
length~$n$ of the path~$S$ are $\delta$-paths. More precisely, $D^n$ is the
event that for all $s\in[n,T]$, we have 
$$
S:[s-n,s]\to\Z \quad t\mapsto S_t
$$
is a $\delta$-path. 

\item[$F^n$:]
Let $F^n$ be the event that in the interval $[-2.45^n,2.45^n]$
a $\delta$-path can generate $w$ only if it ends in the interval
$J$ and does not leave the interval $I$.  
More precisely, $F^n$ is the event that,
For all $\delta$-paths 
$$
R:[0,n]\rightarrow [-2.45^n,2.45^n]
$$
for which $\xi\circ R=w$, we have
$R(n)\in J$ and $R(l)\in[-n,n]$, for all $l\in[0,n]$.

\item[$G^n$:]
Let $G^n$ be the event that the precision of the estimate
based on the $\nu_i$'s is better than 
\begin{equation}
\label{boundforerror2}
\frac{P_\mu(S_r=n+1)-P_\mu(S_r\notin[-n,n+1]}2.
\end{equation}
More precisely, denote the empirical distribution of
$$
\chi(\nu_1+r),\chi(\nu_2+r),\ldots,\chi(\nu_j+r)
$$
where $j$ is the largest $i$ for which $\nu_i\leq T$, by 
$$
\tilde P_\mu(\chi_r=\cdot).
$$
Recall that~$r$ is defined to be
$90\delta n$. Then $G^n$ is the event that, for all $e\in\{1,2,3,4,5\}$, the difference
$$
\abs{\tilde P_\mu(\chi_r=e)-P_\mu(\chi_r=e)}
$$
is strictly less than the expression given in~\ref{boundforerror2}.  
\end{description}

The following lemma shows that when all the events $B^n$ through $G^n$ hold, 
then we can reconstruct $\xi(n+1)$ correctly.  

\begin{lemma}
\label{lem:boundForA^n}
Assume that expression \ref{boundforerror2}
is strictly positive. Then,
$$
B^n\cap C^n \cap D^n \cap F^n \cap G^n\subset A^n.
$$
\end{lemma}

\begin{proof}
If $B^n$ holds, then the random walk stays within the interval 
$[-2.45^n,2.45^n]$ for all times up to and including~$T$. However, thanks to the event $F^n$, a 
$\delta$-path within that interval can generate the pattern $w$ only if it stays within 
the interval $[-n,n]$. This means that when $B^n$ and $F^n$ both hold, 
$\tau_i=\nu_i$ for all $\tau_i\leq T$. In this situation, the estimates of 
$P_\mu(\chi_r=e)$ based on the $\tau_i$'s and the $\nu_i$'s are identical, Since 
they only make use of stopping times up to time~$T$.  The event $C^n$ merely ensures 
the occurrence of at least $1.1^n$ stopping times by time~$T$. When $G^n$ 
holds, we have
$$
\abs{\tilde{P}(\chi_r=e)-P_\mu(\chi_r=e)}
$$
is less than \ref{boundforerror2} and hence for the estimate based on
the $\tau_i$'s we  have that
\begin{equation}
\label{expressione}
\abs{\hat{P}(\chi_r=e)-P_\mu(\chi_r=e)}
\end{equation}
is also less than \ref{boundforerror2}. But this is enough to 
make the reconstruction algorithm work.
To see this, let $a$ be the number
$$
a:=\frac{P_\mu(S_r=n+1)+P_\mu(S_r\notin[-n,n+1]}{2}.
$$
We have proved that when $B^n$, $C^n$, $D^n$, $F^n$ and $G^n$ all hold,
then
$$
\hat{P}(\chi_r=e)-P_\mu(\chi_r=e, S_r\in[-n,n])>a
$$
if $$\xi_{n+1}=e$$ and
$$
\hat{P}(\chi_r=e)-P_\mu(\chi_r=e, S_r\in[-n,n])<a
$$
if $\xi_{n+1}\neq e$.
This means that the reconstruction algorithm works
correctly, since it chooses as estimate for the color
$\xi_{n+1}$ the value $e\in\{1,2,3,4,5\}$ which maximizes
$$
\hat{P}(\chi_r=e)-P_\mu(\chi_r=e, S_r\in[-n,n]).
$$
This completes the proof.
\end{proof}

\section{High probability of the events $B^n$, $C^n$, $D^n$, $F^n$ and $G^n$}

As a consequence of Lemma~\ref{lem:boundForA^n}, we see that
$$
P(A^{nc}) \leq P(B^{nc})+P(C^{nc})+P(D^{nc})+P(F^{nc})+P(G^{nc}),
$$
so in order to show that $A^n$ occurs with probability close to~$1$, we need 
only show that each of $B^n$, $C^n$, $D^n$, $F^n$ and $G^n$ occur with 
probability approaching~$1$ as~$n$ becomes large.  We shall do this by showing 
that the probability that each of these events does not occur is finitely summable 
over~$n$.

\subsection{The event $B^n$}
\label{subsec:b^n}

Let $\sigma^2$ be the variance of the increment distribution of the random 
walk~$S$. Conditions~(\ref{conditionS1}) and~(\ref{conditionS2}) ensure that 
$\sigma^2$ is finite.  Then, since $S_t$ is a sum of i.i.d. increments, direct application of 
Kolmogorov's inequality yields
\begin{eqnarray*}
P(B^{nc})
&=& P(\max_{0\leq t\leq T}\abs{S_t} \geq 2.45^n+1) \\
&\leq& \frac{2.4^{2n}\sigma^2}{\bigl(2.45^n+1\bigr)^2}
< \sigma^2 \left( \bigl( 2.4/2.45 \bigr)^2 \right)^n.
\end{eqnarray*}
Hence $P(B^{nc})$ is finitely summable and $P(B^n)\rightarrow1$ exponentially fast as $n\rightarrow\infty$.

\subsection{The event $C^n$}

We define some auxilliary events to assist in proving that $C^n$ occurs with high probability.  Let $C^n_0$ be the event that the random walk has exited the interval 
$[-(2.39)^n,(2.39)^n]$ at time~$T$:
$$
C^n_0:= \{S_T\notin[-2.39.^n,2.39^n]\}.
$$

Next, let $C^n_1$ be the event that there are at least $2.3^n$ visits to
the point $k_1=n-n^*$ before the random walk leaves the interval $[-
2.39^n,2.39^n]$ for the first time.

Define $C^n_2$ to be the event that among the first $2.3^n$ visits
to $k_1$ there are at least $1.1^n$ for which the random walk
subsequently takes $n$ steps to the right. That is, if $t_i$ denotes
the time of the $i$-th visit to $k_1$, then, $C^n_2$ is the event
that
$$
\abs{\{i\in\N \st i\leq 2.3^n\;,\;S_{t_i+s}=k_1+s,\ \forall s\in[0,n]\}}
\geq 1.1^n.
$$

It is easy to see that
\begin{equation}
\label{C}
C^n_0\cap C^n_1\cap C^n_2\subset C^n.
\end{equation}
Thus, in order to prove that $C^n$ occurs with high probability, we need only 
prove that each of the three events $C^n_0$, $C^n_1$ and $C^n_2$ occurs with 
high probability.

\subsubsection*{Proof that $C^n_0$ occurs with high probability}

Let $\rho$ denote the third absolute moment of the increment distribution of~$S$ 
and note that, as was the case for~$\sigma^2$, conditions~(\ref{conditionS1}) and~(\ref{conditionS2}) guarantee 
$\rho<\infty$.  Then, fixing $b=2.39/2.4$, we have
\begin{eqnarray*}
P(C_0^n)
&=& P(S_T \notin [-2.39^n, 2.39^n])
= P\left( \frac{S_{2.4^{2n}}}{2.4^n\sigma} \notin [-b^n, b^n]/\sigma \right) \\
&=& 1-\Phi_{2.4^{2n}}\left(b^n/\sigma\right) + \Phi_{2.4^{2n}}\left(-b^n/\sigma\right),
\end{eqnarray*}
where $\displaystyle \Phi_t(x) := P\left(
\frac{S_t}{\sigma t^{1/2}} \leq x\right)$.  Let~$\Phi$ denote
the distribution function of a standard Gaussian random variable.  
Now, $\Phi_t\weaklyto Phi$ as $t\to\infty$ by the central limit
theorem. Since $b^n\rightarrow 0$, we see that
$$
\lim_{n\to\infty} P(C_0^n) = 1.
$$
Furthermore,
\begin{eqnarray*}
1-P(C_0^n)
&=& \Phi_T(B^n/\sigma)-\Phi_T(-b^n/\sigma) \\
&\leq& \Phi(B^n/\sigma)-\Phi(-b^n/\sigma) + \abs{\Phi_T(B^n/\sigma)-\Phi(b^n/\sigma)} + \\
&& \qquad \abs{\Phi_T(-B^n/\sigma)-\Phi(-b^n/\sigma)}.
\end{eqnarray*}
By the Berry-Ess\'een theorem,
$$
\abs{\Phi_t(x)-\Phi(x)} 
\leq \frac{\beta\rho}{\sigma^3 t^{1/2}},
$$
for all $x\in\R, t\geq1$, where $\beta=0.7655$.  Therefore,
\begin{eqnarray*}
1-P(C_0^n)
&\leq& 2\int_0^{b^n/\sigma} (2\pi)^{-1/2}e^{-u^2/2}\,du +\frac{2\beta\rho}{2.4^n\sigma^3}.
\end{eqnarray*}.
Since $\int_0^h (2\pi)^{-1/2}e^{-u^2/2}\,du = \frac1{\sigma\sqrt{2\pi}}h + o(h^3)$, we then have
$$
1-P(C_0^n) 
\leq 2(2\pi\sigma^2)^{-1/2}b^n+o(b^{3n}) 
+ \frac{2\beta\rho}{2.4^n\sigma^3}.
$$
When~$n$ is large enough, the right-hand side will be bounded by
$2(2\pi)^{-1/2}\sigma^{-1}b^{n/2}$ and so $P(C_0^n)$ converges
exponentially quickly to~$1$.

\subsubsection*{Proof that $C_1^n$ occurs with high probability}
	
Let us imagine for a moment that the random walk~$S$ is simple, symmetric
and starts at $k_1+1$. Let $\eta$ be the first time
that~$S$ hits $k_1$ or $2.39^n$:
$$
\eta=\min\{t\geq 0 \st S_t\in\{k_1,2.39^n\}\;\}.
$$
The random walk is a Martingale and hence we have that
\begin{equation}
\label{k1S0}
k_1+1 = \E[S_0]=\E[S_\eta]=k_1\cdot P(S_\eta=k_1)+
(2.39^n)P(S_{\eta}=2.39^n)
\geq 2.39^nP(S_\eta=2.39^n).
\end{equation}
Hence, $P(S_{\eta}=2.39^n)\leq \frac{k_1+1}{2.39^n}$ and the simple random walk 
starting at $k_1+1$ has a probability of less then $(k_1+1)(2.39)^{-n}$ of 
hitting $2.39^n$ before hitting~$k_1$. The same bound holds for a simple random 
walk starting at $k_1-1$ and its  probability of hitting $-2.39^n$ before 
returning to $k_1$.  We note that it is possible to obtain a slightly better 
bound, but this suffices here. Hence the number of visits to the point $k_1$ by 
the random walk before hitting either $-2,39^n$ or $2.39^n$ is a geometric 
random variable with parameter $p\leq (k_1+1)2.39^{-n}$. It follows that the 
probability of  making fewer than $2.3^n$ visits to the point $k_1$ before 
hitting the set $\{-2.39^n,2.39^n\}$ is less than 
$(k_1+1)\left(\frac{2.3}{2.39}\right)^n$. This imposes a bound on $P(C^{nc}_1)$ 
that is negatively exponentially small in~$n$ if the random walk is simple. So 
we need to explain why the same kind of result holds for our random walk~$S$ 
whose increment distribution has exponentially decaying tails.

This is done by considering a random walk~$\bar{S}$, which is generated by~$S$ 
and which makes steps of length at most $n^2$. Whenever the random walk $S$ 
jumps further than this, $\bar{S}$ does not move. Hence, $S_0=\bar{S}_0$ and for all $t\in\N$,
$$
\bar{S}_t-\bar{S}_{t-1}
:= \left\{\begin{array}{ll}
S_t-S_{t-1}, & \mbox{if } \abs{S_t-S_{t-1}}\leq n^2, \\
0, & \mbox{otherwise.}
\end{array}\right.
$$
The key here is that the random walks~$S$ and~$\bar{S}$ will very likely be 
identical within a time frame of interest.  To capture this, we introduce 
$C^n_{11}$ which is the event that
$$
S_t=\bar{S}_t\;,\;\forall t\leq 7^n.
$$
Let $C^n_{12}$ be the event that the random walk $\bar{S}$ leaves the interval
$[-2.39^n,2.39^n]$ no later than 
$7^n$.\\
We define the stopping times $\bar{t_i}$ inductively.
Let $\bar{t}_1=0$ and $\bar{t}_{i+1}$
be the first time no earlier than $\bar{t}_i+n^4$ that the random
walk $\bar{S}$ visits the interval $-[n^2,n^2]$:
$$
\bar{t}_{i+1}:=\min\{t\geq \bar{t}_i+n^4 \st \bar{S}_t\in[-n^2,n^2]\}.
$$
Let $C^n_{13}$ be the event that at least $2.38^n$ stopping times $\bar{t}_i$ 
occur before $\bar{S}$ leaves the interval $[-2.39^n,2.39^n]$. \\
Let $\{Y_i\}_1^\infty$ be a sequence of Bbernoulli random variables marking visits to the point $k_1$ at specific stopping times:
 $Y_i:=1$ iff $\bar{S}_{\bar{t}_i+n^4}=k_1$. \\
 finally, define $C^n_{14}$ to be the event that
$$
\sum_{i=1}^{2.38^n}Y_i\geq 2.3^n.
$$

Note that
\begin{equation}
\label{CCC}
C^n_{11}\cap C^n_{12}\cap C^n_{13}\cap C^n_{14}\subset C^n_1.
\end{equation}
The following argument explains why this is so.  Due to $C^n_{12}$, the random walk leaves the interval $[-2.39^n,2.39^n]$ before 
time $7^n$ and by $C^n_{13}$ there are at least $2.38^n$ stopping times 
$\bar{t}_i$ before $\bar{S}$ leaves the interval $[-2.39^n,2.39^n]$. Thus, 
$C^n_{12}$ and $C^n_{13}$ together imply that at least $2.38^n$ stopping times 
$\bar{t}_i$ will be seen before time $7^n$. The event $C^n_{14}$ forces at least 
$2.3^n$ of these $2.38^n$ stopping times $\bar{t}_i$ to be followed by a visit 
to the point $k_1$ at a time $n^4$ later. Hence, we have that if $C^n_{12}$, 
$C^n_{13}$ and $C^n_{14}$ all hold, then prior to time $7^n$, the random walk 
$\bar{S}$ visits the point $k_1$ at least $2.3^n$ times before finally leaving 
the interval $[-2.39^n,2.39^n]$. The event $C^n_{11}$ stipulates that $S$ and 
$\bar{S}$ be identical up to time $7^n$.  Consequently, $S$ also visits $k_1$ at 
least $2.3^n$ times before leaving the interval $[-2.39^n,2.39^n]$ and so 
$C^n_1$ holds.

From the implication \ref{CCC} it follows that
$$
P(C^{nc}_1)\leq P(C^{nc}_{11})+P(C^{nc}_{12})+P(C^{nc}_{13})+P(C^{nc}_{14}).
$$
To prove finite summability  of $P(C^{nc}_1)$ it is enough to prove it for 
each term on the right-hand side of this inequality.

\paragraph*{\bf Finite summability of $P(C^{nc}_{11})$.}
Note that the event $C^n_{11}$ holds as soon as
the random walk $S$ does not take any step of size larger than $n^2$ up to and including time $7^n$.
Since the tail of the increment distribution of~$S$ is exponentially decaying, we have
$$
P(\abs{S_t-S_{t-1}}>n^2)\leq c_1e^{-c_2n^2},
$$
where $c_1,c_2>0$ are constants not depending on $n$.  Thus,
$$
P(C^{nc}_{11})
\leq c_17^ne^{-c_2n^2}
$$
Provided $c_2$ is large enough, the expression on the right-hand side of this inequality
is indeed finitely summable.

\paragraph*{\bf Finite summability of $P(C^{nc}_{12})$.}
To see the finite summability of $P(C^{nc}_{12})$, first realize that
$$
C^n_{12} = \{\max_{0\leq t\leq 7^n}\abs{\bar S_t} > 2.39^n\}
\supset \{\bar S_{7^n} \notin [-2.39^n,2.39^n]\}
$$
and then emulate the proof that $P(C^{nc}_0)$ is finitely summable.

\paragraph*{\bf Finite summability of $P(C^n_{13})$.}
At any stopping time $\bar{t}_i$, the random walk~$\bar S$ will be in the 
interval $[-n^2,n^2]$. Since~$\bar S$ can make steps of size no larger than 
$n^2$, it can travel a maximum distance of $n^6$ from where it started in 
elapsed time $n^4$.  At time $\bar{t}_i+n^4$, $\bar S$ will therefore be in $[-
n^2-n^6,n^2+n^6]$.  If $\bar S\in[-n^2,n^2]$ at time $\bar t_i+n^4$, then $\bar t_{i+1}=\bar t_i+n^4$ and so it won't have left $[-2.39^n,2.39^n]$ by $\bar t_{i+1}$.  On the other hand,
suppose that~$\bar{S}$ starts at a point $x_0\in 
(n^2,n^2+n^6]$ and let $\eta$ be the first time that it enters 
$[2.39^n,2.39^n+n^2]$ or $[0,n^2]$.  Then
\begin{equation}
\label{x0}
x_0=E_{x_0}[\bar{S}_0]=E_{x_0}[\bar{S}_{\eta}]
=x_2P_{x_0}(\bar{S}_\eta\in[2.39^n,2.39^n+n^2])+x_1P_{x_0}(\bar{S}_\eta\in[0,n^2]),
\end{equation}
where $x_2$ is the conditional expectation of $\bar{S}_\eta$ given
that $\bar{S}_\eta\in[2.39^n,2.39^n+n^2]$
and $x_1$ denotes the conditional expectation
of $\bar{S}_\eta$ given that $\bar{S}_\eta\in[0,n^2]$.
From equation \ref{x0} it follows that
$$
x_0\geq x_2P(\bar{S}_\eta\in[2.39^n,2.39^n+n^2])
$$
and, since $x_2\geq 2.39^n$ and $x_0\leq n^2+n^6$, we have
\begin{equation}
\label{jenesais}
P(\bar{S}_\eta\in[2.39^n,2.39^n+n^2])\leq \frac{n^2+n^6}{2.39^n}.
\end{equation}
An identical argument yields the same bound for the random walk hitting $[-
2.39^n-n^2,-2.39^n]$ before $[-n^2,0]$ given that $\bar S$ is in $[-n^2-n^6,-
n^2]$ at time $\bar t_i+n^4$. Hence for all stopping times $\bar t_i$, the 
expression on the right-hand side of \ref{jenesais} also serves to bound the 
probability of the random walk leaving $[-2.39^n,2.39^n]$ before paying a visit 
to $[-n^2,n^2]$ after time $\bar{t}_i+n^4$. This implies that the number of 
stopping times $\bar{t}_i$ appearing before~$\bar{S}$ exits $[-2.39,2.39^n]$ is 
a geometric random variable with parameter $p\leq (n^2+n^6)/2.39^n$. Hence, the 
probability $P(C_{13}^{nc})$ of seeing fewer than $2.38^n$ stopping times $\bar{t}_i$ before the 
random walk leaves the interval $[-2.39^n,2.39^n]$ is less than
$$
\left((n^2+n^6)\frac{2.38}{2.39}\right)^n.
$$
This shows that $P(C^{nc}_{13})$ is finitely summable.

\paragraph*{\bf Finite summability of $P(C^n_{14})$.}
Let $\sF_i$ be the $\sigma$-algebra generated
by $\bar{S}_0,\bar{S}_1,\ldots,\bar{S}_{\bar{t}_i}$. Then by 
the Local Central Limit Theorem, we find that for $n$ large enough,
there exists a constant $c_4>0$ (not depending on $n$)  such that
$$
P(\bar{S}_{\bar{t}_i+n^4}=k_1 \given \sF_i)\geq \frac{c_4}{n^2}
$$
almost surely. (We use the fact that at time $\bar{t}_i$ the random
walk $\bar{S}$ is located in the interval $[-n^2,n^2]$.)
This then implies that the sum
$$
\sum_{i=1}^{2.38^n}Y_i
$$
is bounded below by
$$
\sum_{i=1}^{2.38^n}Y_i^*
$$
where $Y_i^*$ are i.i.d. Bernoulli random variables with parameter $c_4/n^2$. 
Thus, we have 
\begin{equation}
\label{CN4}
P(C^{nc}_{14})
\leq P\left(\sum_{i=1}^{2.38^n}Y_i^*<2.3^n\right)
=P\left(\sum_{i=1}^{2.38^n}(1-Y_i^*) \geq 2.38^n-2.3^n+1\right).
\end{equation}
Note that the $(1-Y_i^*)$'s remain i.i.d. Bernoulli but have parameter $1-c_4/n^2$.
The second non-central moment of 
$$
\sum_{i=1}^{2.38^n}(1-Y_i^*)
$$
is equal to
$2.38^n(1-c_4/n^2)$.  Applying Chebychev's inequality without centring yields
$$
P\left(\sum_{i=1}^{2.38^n}(1-Y_i^*) \geq 2.38^n-2.3^n+1\right)
\leq \frac{2.38^n\bigl(1-c_4/n^2\bigr)}{\left(2.38^n-2.3^n\right)^2}
\leq 2.38^{-n} (2.38/0.08)^2,
$$
which is finitely summable over~$n$. From Equation \ref{CN4}, this expression 
also bounds $P(C^{nc}_{14})$.

\subsubsection*{\bf Proof that $C^n_2$ occurs with high probability}

Recall that $t_i$ denotes the $i$-th visit by the random walk $S$ to the point 
$k_1$. We will define a subset $\{\kappa_1,\kappa_2,\ldots\}$ of the $t_i$'s 
inductively as follows. Fix $\kappa_1:=t_1$ and, for $i\geq1$, let 
$\kappa_{i+1}$ be the first $t_j$ after time $\kappa_i+n$:
$$
\kappa_{i+1}:=\min\{t_j\geq \kappa_i+n \st j\in\N\}.
$$
The $\kappa_i$'s form a strictly increasing sequence.  Note that, among the first $2.3^n$ stopping times $t_i$, there are at least 
$(2.3)^n/n$ stopping times $\kappa_i$. Let $Y_i$ be a Bernoulli variable which 
is equal to one iff the random walk takes $n$ steps to the right immediately 
following time $\kappa_i$: 
$$
Y_i:=\ind{S_{t_i+s}=k_1+s, \forall 0<s\leq n}.
$$ 
The variables $Y_1,Y_2,\ldots$ are i.i.d. and
$P(Y_i=1) = \alpha^n$, 
where $(1-\epsilon)/2 \leq \alpha \leq 1/2$.
Let $C^n_3$ be the event 
$$
\sum_{i=1}^{2.3^n/n}Y_i\geq 1.1^n.
$$
Then since among the first $2.3^n$ stopping times $t_i$ there
are at least $2.3^n/n$ stopping times $\kappa_i$, we see that
$C^n_3$ implies $C^n_2$.
Together with Chebychev's inequality, this yields
\begin{eqnarray*}
P(C^{nc}_2) 
&\leq& P(C^{nc}_3)
= P\left( \sum_{i=1}^{2.3^n/n}Y_i \geq 1.1^n \right) \\
&=& P\left( \sum_{i=1}^{2.3^n/n} \bigl(Y_i-\E[Y_i]\bigr) \geq 1.1^n-2.3^n\alpha^n/n \right) \\
&\leq& P\left( \sum_{i=1}^{2.3^n/n} \bigl(Y_i-\E[Y_i]\bigr) \geq 1.1^n-1.15^n/n \right) \\
&\leq& \frac{2.3^n \cdot 0.5^n/n}{(1.15^n/n-1.1^n)^2}
=1.15^{-n}\; n \left(1-n(1.1/1.15)^n \right)^{-2}.
\end{eqnarray*}
Since $n(1.1/1.15)^n \downarrow 0$ as $n\rightarrow\infty$,
$P(C^{nc}_2) \leq 529n\cdot 1.15^{-n}$ and so $P(^{nc}_2)$ is indeed finitely summable.

\subsection{The event $D^n$}

Let $R:[0,n] \rightarrow \Z$ be a piece of length~$n$ taken from the sample path of~$S$. In order
to prove $D^n$ occurs with high probability, we must first obtain a bound on the probability
of~$R$ being a $\delta$-path. towards this end, define
$$
D^n_0 := \{ R:[1,n] \rightarrow \Z \mbox{ is a $\delta$-path}\}.
$$
Also, set $X_i:=\abs{R_{i+1}-R_i} \cdot \ind{\abs{R_{i+1}-R_i}\neq1}$ and 
$Y_i:=\ind{\abs{R_{i+1}-R_i}\neq1}$.  The $X_i$'s represent the sizes of jumps 
over distances greater than unity while the $Y_i$'s indicate those times at 
which a non-unit jump occurred.  Note that $X_i=Y_i=0$ whenever the $i$-th jump 
is of unit length.  Observe that by definition, $D^n_0 = D^{nc}_1 \cap 
D^{nc}_2$, where
$$
D^n_1 := \left\{ \sum_{i=1}^{n-1} X_i > \delta n \right\}
$$
and
$$
D^n_2 := \left\{ \sum_{i=1}^{n-1} Y_i > \delta n \right\}.
$$
Therefore, $P(D^{nc}_0) \leq P(D_1^n) + P(D_2^n)$.
Our aim is to show that $P(D_1^n)$ and $P(D_2^n)$ can be bounded exponentially small as~$n$
becomes large.

From the moment generating function of $X_i$, we have
\begin{eqnarray*}
\E[e^{sX_i}]
&=& P(X_i=0) + \sum_{j=2}^\infty e^{sj} P(X_i=j) \\
&\leq& \epsilon+ \epsilon \sum_{j=2}^\infty e^{sj}e^{-cj} \\
&=& \epsilon + \epsilon \sum_{j=2}^\infty \left(e^{s-c}\right)^j \\
&=& \epsilon \left(1+\frac{e^{-2(c-s)}}{1-e^{-(c-s)}} \right).
\end{eqnarray*}
Applying Churnov's inequality, we have
\begin{eqnarray*}
P(D_1^n)
&=& P\left( \sum_{j=0}^{n-1} X_j \geq \delta (n+1)\right) \\
&\leq& \E\left[ e^{s(\sum_{j=0}^{n-1}X_j-\delta (n+1))} \right] \\
&=& e^{-\delta (n+1)s} \left( \epsilon \left( 1+\frac{e^{-2(c-s)}}{1-e^{-(c-s)}} \right) \right)^n \\
\leq \left( \epsilon \left( 1 + \frac{e^{-2(c-s)}}{1-e^{-(c-s)}} \right) \right)^n,
\end{eqnarray*}
for all $s>0$.
Since $\frac{e^{-2(c-s)}}{1-e^{-(c-s)}}$ is strictly increasing on $[0,\infty)$, we see that
$$
P(D^n_1)
\leq \left( \epsilon \left( 1 + \frac{e^{-2c}}{1-e^{-c}} \right) \right)^n \\
= e^{-c^+n},
$$
where $c^+:=-\ln\left(1+\frac{e^{-2c}}{1-e^{-c}}\right) -
\ln\epsilon$. Observe that $c^+$ can always be made positive and as
large as we like by choosing $\epsilon$ sufficiently small. In
particular, $\epsilon$ can be chosen to ensure that $c^+ >
c':=c+2\ln2.4+\ln2$, whence $P(D_1^n) \leq e^{-c'n}$ for all~$n$.

Next,
$$
\E[e^{sY_i}] = e^s P(\abs{R_{i+1}-R_i}\neq1) \leq \epsilon e^s.
$$
Applying the same standard large deviation argument used above, we obtain
\begin{eqnarray*}
P(D^n_2)
&=& P\left( \sum_{j=0}^{n-1} Y_i \geq \delta (n+1)\right) \\
&\leq& \E\left[ e^{s\sum_{j=0}^{n-1} Y_i -\delta (n+1)s} \right] \\
&\leq& e^{-\delta(n+1)s} \left( \E(e^{sY_1})\right)^n 
\leq e^{ns} \epsilon^n,
\end{eqnarray*}
for all $s>0$.  Since $e^ns$ is strictly increasing in~$s$, we have $P(D_2^n) \leq \epsilon^n
\leq e^{-c'n}$, provided $\epsilon$ is once again chosen small enough.

Now, the probability of seeing a $\delta$-path in a piece of~$S$ of length~$n$ can be bounded
below by
$$
P(D^n_0) \geq 1 - P(d_1^n)-P(D_2^n)
\geq 1 - 2e^{-c'n} = 1-e^{-c''n},
$$
where $c''=c+2\ln2.4$.
Therefore, $P(D_0^n)\nearrow1$ exponentially fast as $n\to\infty$.

Now, $D^n$ may be expressed as the intersection
$$
d^n = \bigcap_{s=n}^T D^{n,s},
$$
where $D^{n,s}$ is the event that the segment 
$$
S:[s-n,s]\rightarrow\Z
$$
of the sample path of~$S$ is a $\delta$-path.  Then,
\begin{eqnarray*}
P(D^{nc})
&\leq& \sum_{s=n}^T P\bigl((D^{n,s})^c\bigr) 
= \sum_{s=n}^T P(D^{nc}_0) \\
&\leq& (T-n+1)e^{-c''n} \leq 2.4^{2n} e^{-(c+2\ln2.4)n} = e^{-cn},
\end{eqnarray*}
which shows that $P(D^n)\to1$ exponentially fast.

\subsection{The event $F^n$}

Let $R:[0,n]\rightarrow\Z$ denote a $\delta$-path of length~$n$. Define $F^n_0$ 
to be the event that for all $x\in[-2.45^n,2.45^n]\setminus [-2n-\delta n,3n]$, 
there exists no $\delta$- path~$R$ such that
\begin{equation}
\label{obs=w}
\xi\circ R=w\;{\rm and}\;R(n)=x.
\end{equation}
Define $F^n_1$ to be the event that
for all $x\in [n^*+\delta n,3n]$
there exist no $\delta$-path $R$ satisfying \ref{obs=w} and
let $F^n_2$ denote the event that,
for all $x\in [-2n-\delta n,n^*-21\delta n]$,
there exists no $\delta$-path $R$ satisfying \ref{obs=w}.
Then let $F^n_J$ be the event that for any $\delta$-path
$R:[0,n]\rightarrow [-2.45^n,2.45^n]$ satisfying
$\xi\circ R=w$, we have $R(n)\in J$.  It should be clear that
\begin{equation}
\label{F012J}
F^n_0\cap F^n_1\cap F^n_2\subset F^n_J.
\end{equation}
 Let $F^n_-$ be the event that for any $\delta$-path
$R:[0,n]\rightarrow [-2.45^n,2.45^n]$ satisfying
$\xi\circ R=w$, we have $R(0)\in J^-$
where
$$
J^-:=[n^*-n-\delta n,n^*-n+21\delta n].
$$
Note that for any $\delta$-path $R:[0,n]\rightarrow\Z$ with $R(0)\in J^-$ and $R(n)\in J$, we will have
\begin{equation}
\label{conditions}
n^*-n+(1-\delta)s-2\delta n\leq R(s)\leq n^*-n+(1-\delta)s+2\delta n\;,\;
\forall s\in[0,n].
\end{equation}
Fixing~$\delta$ so that $63\delta<1$, (\ref{conditions}) implies that $R(s)\in[-
n,n]$ for all $s\in[0,n]$. Consequently,
$$
F^n_J\cap F^n_-\subset F^n
$$
and hence
$$
P(F^{nc})\leq P(F^{nc}_J)+P(F^{nc}_-).
$$

The proof that $F^n_-$ occurs with high probability is similar to that for 
$F^n_J$ and we shall leave it to the reader. 

To prove the finite summability of $P(F^{nc}_J)$ we use \ref{F012J} which 
implies that 
$$
P(F^{nc}_J)\leq P(F^{nc}_0)+P(F^{nc}_1)+P(F^{nc}_2).
$$
Hence, we only need to prove high probability of the
events $F^n_0$, $F^n_1$ and $F^n_2$.
The key to proving  this is the following lemma:

\begin{lemma}
\label{deltaPathCounting}
Let~$x$ be a point. For every~$n$, the total number of $\delta$-paths of 
length~$n$ starting at~$x$ is less than
$$
2^{n(1+H(\delta)+2\delta)}
$$
Define a $(k,\delta)$-path of length $kn$ to be a path comprising a total 
of $kn$ steps of which fewer than $\delta n$ are non-unit and for 
which the total variation of the non-unit steps is also less than $\delta n$.  Clearly, all $(1,\delta)$-paths of length~$n$ are $\delta$-paths of length~$n$ and conversely.
Then, the number of $(10\delta,\delta)$-paths of length $10\delta n$ starting at~$x$ is less than
$$
2^{10\delta n(1+H(0.1)+0.3)}.
$$
\end{lemma}

\begin{proof}
First, if all steps are either~$+1$ or~$-1$, we have at most $2^n$ choices, 
since the path makes at most~$n$ jumps of length~$1$. Then among the $n$ steps, 
we must consider which of them are not of unit length. There are at most $\delta 
n$ such steps and, since~$\delta$ is taken to be small (in particular 
$\delta<1/2$), we shall have fewer than $\binom n{\delta n}$ possible 
arrangements of the non-unit jumps among the unit jumps.  An argument involving 
Stirling's approximation reveals the bound $2^{H(\delta)n}$ on the number of 
possible arrangements, where the entropy $H(\delta)$ is given by $H(\delta) := -
\delta\log_2\delta -(1-\delta)\log_2(1-\delta)$.

Next we need to take into account the length of each non-unit step. We know that 
the total variation of the non-unit jumps is no more than $\delta n$. Hence, we 
can partition an interval of length at most $\delta n$ to determine the lengths 
of the non-unit steps which are not of length $0$. This gives a maximum of 
$2^{\delta n}$ possibilities.  Because we can have steps of length zero, we must 
also determine which of the non-unit steps have length zero and this gives 
another $2^{\delta n}$ possibilities at most. Finally, we need to account for 
the signs of the non-unit steps. Since there are only two possible signs, There 
are at most $2^{\delta n} $ choices. Multiplying all the preceding bounds yields 
an upper bound on the maximum number of $\delta$-paths of length~$n$ which start 
at~$x$:
$$
2^{n(1+H(\delta)+3\delta)}.
$$

To compute a bound on the number of $(10\delta,\delta)$-paths of length 
$10\delta n$ starting at~$x$, we proceed as before. Determining the steps of 
size $+1$ and $-1$, we have at most $2^{10\delta n}$ choices. Then we need to 
factor in the steps which are not of unit-size. There are at most $\delta n$ of 
these, which corresponds to a proportion of no more than $0.1$ the length of the 
path. Hence, we have at most $2^{10\delta nH(0.1)}$ choices. Finally, 
determining the length and sign of the non-unit steps as before gives a maximum 
of $2^{3\delta n}$ choices. Combining these bounds in a product, we find that 
there are less than 
$$
2^{10\delta n(1+H(0.1)+0.3)}.
$$
$(10\delta,\delta)$-paths of length $10\delta n$ starting at~$x$.
\end{proof}

\subsubsection*{Proof that $F^n_0$ occurs with high probability}

Let $x$ belong to the set 
$$
X:=[-(2.45)^n,(2.45)^n] \setminus [2n-\delta n,3n].
$$
Now, $n+\delta n-1$ is  the maximum distance
a $\delta$-path can travel in~$n$ steps.
So any $\delta$-path of length $n$ which visits~$x$ 
will stay outside $[-n,n]$. 
Thus, if $R:[0,n]\rightarrow
\Z$ is a non-random
$\delta$-path such that $R(n)=x$, then $\xi\circ R$ is independent
of $w$. To see this recall that The scenery is i.i.d. and so disjoint parts of it are independent
of each other. note that the pattern $w$ is a substring of $\xi_{-n}\xi_{-n+1}\ldots \xi_{n}$
while $\xi\circ R$ only depends on parts of the scenery outside
$[-n,n]$. 
Furthermore, the string $w$ is i.i.d. with each of the five possible colors
appearing with probability $1/5$, that is,
$$
P(\xi\circ R=w)=(1/5)^n,
$$
where~$R$ is a non-random $\delta$-path of length~$n$ ending at~$x$. Since there 
are at most  $2^{n(1+H(\delta)+2\delta)}$ $\delta$-paths of length $n$ ending 
at~$x$, we get
\begin{equation}
\label{inequalityF0}
P(F^{nc}_{0,x}) \leq \left(\frac{2^{1+H(\delta)+2\delta}}{5}\right)^n
\end{equation}
where $F^n_{0,x}$ designates the event that there exists no
$\delta$-path of length $n$ for which $\xi\circ R=w$ and 
$R(n)=x$. Finally, $F^n_0$ may be expressed as
$$
F^n_0=\cap_{x\in X} F^n_{0,x}.
$$
Hence,
\begin{eqnarray}
P(F^{nc}_0)
&\leq& \sum_{x\in X} P(F^{nc}_{0,x}) \nonumber \\
&\leq& \abs X \left(\frac{2^{1+H(\delta)+2\delta}}{5}\right)^n 
= 2\left(\frac{2.45 \cdot 2^{1+H(\delta)+2\delta}}{5}\right)^n,
\label{Fnc0}
\end{eqnarray}
since~$X$ contains fewer than $2\cdot 2.45^n$ elements.
Now, $H(\delta)+2\delta$ converges to zero as~$\delta$ tends to zero. Hence, by taking
$\delta>0$ small enough, we see that
\begin{equation}
\label{frac}
\frac{2.45 \cdot 2^{1+H(\delta)+2\delta}}{5}
\end{equation}
can be made as close as we like to $2.45\cdot2/5=0.98<1$.  Fixing $\delta>0$ 
small enough so that expression \ref{frac} is strictly less than $1$, we have 
that inequality \ref{Fnc0} constitutes a negative exponential bound for 
$P(F^{nc}_0)$. 

\subsubsection*{Proof that $F^n_1$ occurs with high probability}

Let $x$ be a non-random integer such that 
\begin{equation}
\label{x}
x\in [n^*+\delta n,3n].
\end{equation}
Assume also that $R:[0,n]\rightarrow \N$ is
a non-random $\delta$-path such that $R(n)=x$. Note that because~$R$ is 
a $\delta$-path, $R(n-i)$ is never further from~$x$ than $\delta n+i$
and hence because of \ref{x}, we have $R(n-i)>n^*-i$.
It follows that the $n-i$-th letter of the word
$w$, which we denote by $w(n-i)$ and which is equal to
$\xi_{n^*-i}$, is independent of 
$$
\xi(R(n-i))\xi(R(n-i+1))\ldots \xi(R(n)).
$$
This holds for all $i=0,1,\ldots,n$. Hence we find that if
\ref{x} holds with $R(n)=x$ for a non-random $\delta$-path $R$, then
$$
P(\xi\circ R=w)=(1/5)^n.
$$
Next, we synthesize the proof that $F^n_0$ occurs with high probability. First 
let $F^n_{1,x}$ be the event that there exists no $\delta$-path of length $n$ 
ending in $x$ and generating $w$. Since the number of such paths, according to 
Lemma~\ref{deltaPathCounting}, is no greater than $2^{n(1+H(\delta)+2\delta)}$, 
we have
$$
P(F^{nc}_{1,x})\leq \left(\frac{2^{(1+H(\delta)+2\delta)}}{5}\right)^n.
$$
Finally, there are fewer than $3n$ points $x$ in the set $[n^*+\delta n,3n]$ and 
hence the last inequality above yields
\begin{equation}
\label{boundforF1}
P(F^{nc}_{1})\leq 3n\left(\frac{2^{(1+H(\delta)+2\delta)}}{5}\right)^n.
\end{equation}
Once again choosing $\delta>0$ small enough, we have 
$$
\frac{2^{(1+H(\delta)+2\delta)}}{5}<1,
$$
whence inequality \ref{boundforF1} gives the desired negative exponential bound 
on $P(F^{nc}_1)$.

\subsubsection*{Proof that $F^n_2$ occurs with high probability.}

Let $x\in [-2n-\delta n,n^*-21\delta n]$.
Assume that $R:[0,n]\rightarrow \Z$ is a $\delta$-path such that
$R(n)=x$. Note then that for all $i\leq 10\delta n$ we have
$$
R(n-i)\leq x+i+\delta n\leq x+10\delta n+\delta n<n^*-10\delta n.
$$
Hence 
$$
\xi(R(n-10\delta n))\xi(R(n-10\delta n+1)) \xi(R(n-10\delta n+2))
\ldots \xi(R(n))
$$
is independent of 
$$
\xi_{n^*-10\delta n} \xi_{n^*-10\delta n+1}\ldots 
\xi_{n^*}
$$
which corresponds to the last
$10\delta n$ letters of the word~$w$. Hence, we get
$$P(\xi\circ R=w)\leq \left( \frac{1}{5}\right)^{10\delta n}.
$$
By Lemma~\ref{deltaPathCounting}, there are no more than $2^{10\delta 
n(1+H(0.1)+0.3)} $ $\delta$-paths of length $10\delta$ ending in a given point 
$x$. Let $F^n_{2,x}$ be the event that there is no $\delta$-path of length $n$ 
such that $R(n)=x$ and $\xi\circ R=w$.  Then
$$
P(F^{nc}_{2,x})\leq \left(\frac{2^{10\delta (1+H(0.1)+0.3)}}{5}\right)^n
$$
for all~$x\in[-2n-\delta n,n^*-21\delta n]$. Since there are no more than $3n$ points in $[-2n-\delta n,n^*-21\delta n]$, we find
\begin{equation}
\label{boundF2}
P(F^{nc}_2)\leq 3n\left(\frac{2^{ 1+H(0.1)+0.3}}{5}\right)^{10\delta n}.
\end{equation}
Now, $\frac{2^{1+H(0.1)+0.3)}}{5}$ is strictly less than $1$ so
expression \ref{boundF2} provides the desired exponential bound
on $P(F^{nc}_2)$.

\subsection{The event $G^n$.}

We shall use $\tilde\mu$ to denote the empirical distribution of
$$
S_{\nu_1},S_{\nu_2},\ldots,S_{\nu_j},
$$
where the random variable~$j$ is the largest $i$ such that $\nu_i\leq T=2.4^n$; that is,
$$
\tilde\mu(x) := \tilde\mu(\{x\})
= \frac1j \sum_{i=0}^j\ind{S_{\nu_i}=x}.
$$

Let $G^n_1$ denote the event that the difference between~$\tilde\mu$ and~$\mu$ is less than or equal to $n^{3/2}/1.1^{n/2}$
in total variation:
$$
G^n_1 := \left\{ \sum_{x\in[-n,n]}\abs{\tilde\mu(x)-\mu(x)}
< \frac{n^{3/2}}{\sqrt{1.1}^n} \right\}.
$$

Next, set $G^n_2:=B^n\cap C^n\cap D^n\cap F^n$. Clearly, we can write 
$P(G^{nc}_1) \leq P(G^{nc}_1 \given G^n_2) + P(G^{nc}_2)$.  Since $P(G^{nc}_2) 
\leq P(B^{nc})+ P(C^{nc})+P(D^{nc})+P(F^{nc})$ and the 4 quantities on the 
right-hand side are finitely summable, $P(G^{nc}_2)$ is itself finitely summable. 
Therefore, to show that $P(G^{nc}_1)$ is finitely summable, it is enough to 
prove finite summability of $P(G^{nc}_1 \given G^n_2)$.  Towards this end, 
observe that the event $G^n_2$ implies that at least $1.1^n$ stopping times 
$\nu_i$ manifest by time~$T$ and that each of these times marks the end of a 
$\delta$-path of length~$n$ in~$I$ whose final position lies in~$J$.  As a consequence,
$\tilde\mu$ has support on $J\subset I$ where $\abs J = 22\delta n+1 
<n < 2n+1 = \abs I$ for $\delta>0$ sufficiently small. Now,
$\sum_{x\in J}\abs{\tilde\mu(x)-\mu(x)} 
\leq n\sup_{x\in J}\abs{\tilde\mu(x)-\mu(x)}$.
Conditioning on $j=k$, we have
\begin{eqnarray}
\lefteqn{P\left(\sum_{x\in J}\abs{\tilde\mu(x)-\mu(x)} \geq \frac{n^{3/2}}{\sqrt{1.1}^n} \midgiven \{j=k\}\cap G^n_2\right)}
\nonumber \\
&\leq& P\left( \sup_{x\in J}\abs{\tilde\mu(x)-\mu(x)} \geq \frac{\sqrt n}{\sqrt{1.1}^n} \midgiven \{j=k\}\cap G^n_2\right)
\nonumber \\
&=& P\left( \cup_{x\in J} \bigl\{ \abs{\tilde\mu(x)-\mu(x)} \geq \frac{\sqrt n}{\sqrt{1.1}^n} \bigr\} \midgiven \{j=k\}\cap G^n_2\right) 
\nonumber \\
&\leq& \sum_{x\in J} P\left( abs{\tilde\mu(x)-\mu(x)} \geq \frac{\sqrt n}{\sqrt{1.05}^n} \midgiven \{j=k\}\cap G^n_2\right).
\label{tildemu.mu.comp}
\end{eqnarray}
Next, define $Y_0(x)=0$ and $Y_i(x)=\sum_{k=1}^iX_i(x)$, for $i\in\N$, where $X_i(x)=\ind{S_{\nu_i}=x}-\mu(x)$.
Now $Y_i(x)$ is a Martingale with increment size bounded by~$1$ and $Y_k(x)-Y_0(x)=k(\tilde\mu(x)-\mu(x)$.  An application of the Azuma-Hoeffding inequality shows that 
\begin{eqnarray*}
P\left( \tilde\mu(x)-\mu(x) \geq \frac{\sqrt n}{\sqrt{1.1}^n} \midgiven \{j=k\}\cap G^n_2\right)
&=& P\left( Y_k(x)-Y_0(x) \geq \frac{k\sqrt n}{\sqrt{1.1}^n} \midgiven \{j=k\}\cap G^n_2\right) \\
&\leq& e^{-\frac{k^2n}{2k\cdot1.1^n}}
\leq e^{-0.5nk/1.1^n}.
\end{eqnarray*}
Similarly, $P\left( \tilde\mu(x)-\mu(x) \leq -\frac{\sqrt n}{\sqrt{1.1}^n} \midgiven \{j=k\}\cap G^n_2\right) \leq e^{-0.5nk/1.1^n}$ so that
$$
P\left( \abs{\tilde\mu(x)-\mu(x)} \geq \frac{\sqrt n}{\sqrt{1.1}^n} \midgiven \{j=k\}\cap G^n_2\right) \leq 2e^{-0.5nk/1.1^n}.
$$
Combining this bound with \ref{tildemu.mu.comp} then yields
$$
P\left(\sum_{x\in J}\abs{\tilde\mu(x)-\mu(x)} \geq \frac{n^{3/2}}{\sqrt{1.1}^n} \midgiven \{j=k\}\cap G^n_2\right)
\leq 2\abs J e^{-0.5nk/1.1^n}
< 2ne^{-0.5nk/1.1^n}.
$$
Since $C^n$ guarantees $j\geq 1.1^n$, $k$ will belong to $[1.1^n,T]$ and we find that
\begin{eqnarray*}
P(G^{nc}_1 \given G^n_2)
&=& \sum_{k=1.1^n}^T P(G^{nc}_1 \given \{j=k\}\cap G^n_2) P(j=k \given G^n_2) \\
&\leq& \sum_{k=1.1^n}^T 2ne^{-0.5nk/1.1^n} P(j=k \given G^n_2) 
\leq 2ne^{-n/2},
\end{eqnarray*}
which is finitely summable.

Now, if it could be shown that $G^n_1\cap G^n_2 \subset G^n$, then we would have 
$P(G^{nc}) \leq P(G^{nc}_1) + P(G^{nc}_2)$.  However, we have already shown that 
the two terms on the right-hand side are finitely summable, whence $P(G^{nc})$ 
would be finitely summable. 

So, it remains to show that $G^n_1\cap G^n_2 \subset G^n$ for~$n$ large enough.  In order to do this, we need to show that 
\begin{equation}
\label{90delta}
\frac{P_\mu(S_r=n+1)-
P_\mu(S_r\notin[-n,n+1])}{2}\geq \left(\frac{1-\epsilon}2\right)^{90\delta n+2}. 
\end{equation}
Recall  that~$r$ is defined to be $90\delta n$.  We shall take 
$\delta>0$ small enough so that $\left(2/(1-
\epsilon)\right)^{90\delta}<\sqrt{1.1}$. Then for all $n$ large enough, the 
bound on the right side of \ref{90delta} is strictly larger than 
$n^{3/2}/1.1^{n/2}$. However, the event $G^n_1$ guarantees $n^{3/2}/1.1^{n/2}$ 
as an upper bound on the difference in total variation between the measure 
$\tilde\mu$ and~$\mu$. In turn, the difference in total variation 
between~$\tilde\mu$ and~$\mu$ bounds the difference between 
$P_{\tilde\mu}(\chi_r=e)$ and $P_\mu(\chi_r=e)$ for every $e\in\{1,2,3,4,5\}$. 
Thus, we have:
\begin{equation}\label{onemoretime}
\abs{P_{\tilde\mu}(\chi_r=e)-P_\mu(\chi_r=e)} \leq \sum_x\abs{\tilde\mu(x)-
\mu(x)} \leq \frac{n^{3/2}}{\sqrt{1.1}^n}.
\end{equation}
Note that $\tilde 
P_\mu(\chi_r=e)=P_{\tilde\mu}(\chi_r=e)$.  Therefore, if (\ref{90delta}) holds, 
$\abs{\tilde P_\mu(\chi_r=e)-P_\mu(\chi_r=e)}$ will be bounded by the quantity in 
(\ref{boundforerror2}) for~$n$ large enough and hence $G^n$ will hold.

To finish, we return to the question of (\ref{90delta}) holding.
Since $G^n_2$ holds, the support of $\mu$ is in the interval $J$. Let us
first assume that the random walk $S$ is aperiodic and symmetric with the increment distribution given in (\ref{simplestincrementdistribution}).
We will deal with the more general case later.
Now, $S$ is able to reach the point $n+1$ from any point
of $J$ in $r$ steps with probability
greater than or equal to $\left(\frac{1-\epsilon}2\right)^r=\left(\frac{1-\epsilon}2\right)^{90\delta n}$.
The minimum distance from any point $x\in J$ to the point $n+1$ is $60\delta n$.  Using (\ref{decayrate}) with $b=60
\delta n$ and $q=1.5$, we calculate
the decay in the tail of the state probability distribution beyond a distance of $60\delta n$ from~$x$ after $90\delta n$ steps to be
$$
\rho(90\delta n,60\delta n)=\frac{P(S_r=n+2)}{P(S_r=n+1)}
< \frac{1.5-1}{1.5+1} = \frac15.
$$
Hence, for any point $x\in J$, we have
\begin{eqnarray*}
P_x(S_r=n+1)-P_x(S_r>n+1)-P_x(S_r<-n)
&\geq& (1-0.25-0.25)P_x(S_r=n+1) \\
&=& 0.5 P_x(S_r=n+1),
\end{eqnarray*}
but as mentioned above, $P_x(S_r=n+1)\geq \left((1-\epsilon)/2\right)^r=\left((1-\epsilon)/2\right)^{90\delta n}$
and hence we obtain
$$
P_x(S_r=n+1)-P_x(S_r\notin[-n,n+1])\geq \left(\frac{1-\epsilon}2\right)^{90\delta n+1}.
$$
Here $P_x(\cdot)$ refers to the distribution when~$S$ is started at the point~$x$.
As the event $G^n_2$ holds, the measure $\mu $has support on the interval $J$ and so
$$
P_\mu(S_r=n+1\given F^n) - P_\mu(S_r\notin[-n,n+1]\given F^n)
\geq \left(\frac{1-\epsilon}2\right)^{90\delta n+1}.
$$
Dividing both sides by two then yields (\ref{90delta}).

While the veracity of (\ref{90delta}) has been discussed assuming the aperiodic, 
symmetric random walk of Section~\ref{sec:ex2}, it is not too difficult to show 
that it holds for any random walk satisfying Conditions~(\ref{conditionS1})--
(\ref{conditionS4}) given $G^n_2$, provided we select the parameters of the 
random walk appropriately.  In other words, for large~$n$, $G^n_1\cap 
G^n_2\subset G^n$ when $\epsilon>0$ is sufficiently small and $c>0$ is 
sufficiently large.

\section{Reconstruction of the whole scenery}
\label{wholereconstruction}

To conclude, we prove that we can reconstruct the whole scenery $\xi$ a.s. up to 
equivalence. for this we use the following lemma which was proved in 
\cite{Lowe-Matzinger-Merkl2001}.

\begin{lemma}
Assume that there exists an algorithm which reconstructs $\xi$ with probability 
strictly greater than $1/2$. Then we can also reconstruct $\xi$ with probability 
one. More precisely, if there exists a measurable map
$$
\sA^* : \{1,2,3,4,5\}^\N\rightarrow
\{1,2,3,4,5\}^\Z
$$
such that
$$
P(\sA^*(\xi)\approx \xi)>1/2,
$$
then there exists a measurable map 
$$
\sA:\{1,2,3,4,5\}^\N \rightarrow \{1,2,3,4,5\}^\Z
$$
such that
$$
P(\sA(\xi)\approx \xi)=1.
$$
\end{lemma}

So, we need to build a map which reconstructs $\xi$ with probability strictly 
greater than $1/2$. Due to inequality \ref{finitesummability}, there exists a 
non-random $n_0$ such that
\begin{equation}
\label{itworksaftern0}
P\left( \hat{\xi}(n+1)=\xi(n+1), \hat{\xi}(-n-1)=\xi(-n-1),\forall n\geq n_0
\right)\geq \frac45.
\end{equation}
We tune the parameter $\epsilon>0$ small enough so that the random walk $S$ only 
makes $\pm1$ steps for a long time. It is known that we can reconstruct a finite 
piece of $\xi$ close to the origin (see \cite{Lowe-Matzinger-Merkl2001}) 
with probability as close to 
one as we like when we are dealing with a simple random walk on a five-color 
scenery provided we have enough observations. So by taking $\epsilon>0$ small 
enough we can ensure that the random walk follows a simple random walk path long 
enough to reconstruct the finite piece $\xi\vert_{[-n_0,n_0]}$ with probability 
larger than $\frac45$. So, we start by reconstructing the finite piece 
$\xi\vert_{[-n_0,n_0]}$. Then we proceed inductively in~$n$ for $n\geq n_0$. 
Once we have reconstructed $\xi\vert_{[-n,n]}$ we estimate $\xi(n+1)$ and $\xi(-
n-1)$ using the algorithm described in Section~\ref{thealgorithm}. In this way 
we end up reconstructing $\xi$ correctly with probability at least $3/5$ which 
is strictly larger than $1/2$. This establishes that we can reconstruct $\xi$ 
with probability one and thus completes the proof of Theorem~\ref{maintheorem}.

One more remark: Since in general one can not reconstruct the scenery but only reconstruct it
up to equivalence, it follows that we can not reconstruct $\xi_{|[-n,n]}$ exactly. Instead,
from the result of matzinger \cite{Lowe-Matzinger-Merkl2001}, we 
have that we can  only perform reconstruction
successfully with high probability on an approximately centered interval $\xi_{|[-n+l,n+l]}$
where $l$ is small compared to $n$, but not known to us. Note however that our one-point
reconstruction algorithm works if we are given the string $\xi(-n+l)\xi(-
n+l+1)\ldots\xi(n+l)$ without knowing its exact position (i.e. not knowing $l$), but only
that $l$ is of order smaller than $n$. 

\section*{Acknowledgements}

The first and third author would like to 
thank the Millennium Nucleus in Information and 
Randomness P04-069-F, MSI, as well as SFB701 in
Bielefeld for supporting this work. The third 
author would like to also thank FAPESP for supporting his visit
to the IME-USP during which time an important
part of this article was written.

\bibliographystyle{plain}
\bibliography{scen08}
 
\end{document}